\documentclass[letterpaper,11pt]{amsart}
\usepackage{amsmath,amssymb,amsthm,pinlabel,tikz,hyperref,mathrsfs,color, thmtools}
\usepackage{verbatim}
\usepackage{float}
\usepackage{caption}
\usepackage{subcaption}
\usepackage{enumitem}
\newcommand{\nc}{\newcommand}
\nc{\dmo}{\DeclareMathOperator}
\dmo{\ra}{\rightarrow}
\dmo{\Prob}{\mathbb{P}}
\dmo{\E}{\mathbb{E}}
\dmo{\N}{\mathbb{N}}
\dmo{\Z}{\mathbb{Z}}
\dmo{\Q}{\mathbb{Q}}
\dmo{\R}{\mathbb{R}}
\dmo{\C}{\mathcal{C}}
\dmo{\X}{\mathcal{X}}
\dmo{\U}{\mathcal{U}}
\dmo{\T}{\mathcal{T}}
\dmo{\F}{\mathcal{F}}
\dmo{\AC}{\mathcal{AC}}
\dmo{\w}{\omega}
\dmo{\MIN}{\mathcal{MIN}}
\dmo{\Mod}{Mod}
\dmo{\PMod}{PMod}
\dmo{\PMF}{\mathcal{PMF}}
\dmo{\Mat}{Mat}
\dmo{\Hdim}{Hdim}
\dmo{\supp}{supp}
\dmo{\UE}{\mathcal{UE}}
\dmo{\vol}{vol}
\dmo{\B}{B}
\dmo{\PB}{PB}
\dmo{\PR}{PSL(2,\mathbb{R})}
\dmo{\GL}{GL(k, \mathbb{C})}
\dmo{\SL}{SL(2, \mathbb{Z})}
\dmo{\Isom}{Isom}
\dmo{\RP}{\mathbb{R} \mathrm{P}}
\dmo{\I}{\mathcal{I}}
\dmo{\el}{\ell_{\C}}
\dmo{\NN}{\mathcal{N}}
\dmo{\Hd}{Hdim}
\dmo{\rk}{rank}
\dmo{\tr}{tr}
\dmo{\llangle}{\langle\langle}
\dmo{\rrangle}{\rangle\rangle}
\dmo{\Unif}{Unif}
\dmo{\Out}{Out}
\dmo{\sumRho}{\mathcal{N}}
\dmo{\stopping}{\vartheta}
\dmo{\diam}{\operatorname{diam}}

\usetikzlibrary{decorations.markings, patterns}
\tikzset{->-/.style={decoration={
  markings,
  mark=at position #1 with {\arrow{>}}},postaction={decorate}}}

\nc{\nt}{\newtheorem}

\nt{theorem}{Theorem}

\newtheorem{thm}{{\bf Theorem}}[section]

\newtheorem{lem}[thm]{{\bf Lemma}}
\newtheorem{cor}[thm]{{\bf Corollary}}
\newtheorem{prop}[thm]{{\bf Proposition}}
\newtheorem{fact}[thm]{Fact}

\newtheorem{claim}[thm]{Claim} 
\newtheorem{remark}[thm]{Remark}

\newtheorem{dfn}[thm]{Definition}
\numberwithin{equation}{section}
\newtheorem{obs}[thm]{Observation}

\title[Hausdorff dimension and sublinearly conical limit set]{Hausdorff dimension of sublinearly conical Myrberg limit set}

\date{\today}
\author{Inhyeok Choi}

\address{%
		School of Mathematics, KIAS\\
		85 Hoegi-ro, Dongdaemun-gu, Seoul 02455, South Korea
}
\email{
        inhyeokchoi48@gmail.com
        }

\begin{document}
\begin{abstract}

Let $\Gamma$ be a non-elementary, non-convex-cocompact Kleinian group acting on $\mathbb{H}^{d}$. We show that the Hausdorff dimension of the sublinearly conical Myrberg limit set of $\Gamma$ is equal to the critical exponent of $\Gamma$. This gives a different proof of a theorem by M. Mj and W. Yang. Our approach is a variant of M. Mj and W. Yang's technique in the direction of the Patterson--Sullivan theory.


\noindent{\bf MSC classes:} 20F67, 30F60, 57K20, 57M60, 60G50
\end{abstract}

\maketitle
{\small
\begin{quote}
The algorithms that generate the other fractals are typically extraordinarily short, as to look positively dumb.

- B. B. Mandelbrot. Proc. R. Soc. Lond. A. 423, 3--16 (1989). 
\end{quote}
}

%%%%%%%%%%%%%%%%%%%%%%%%%%%%%%%%%%%%%%%%%%%%%%%%
%
%							Introduction
%
%%%%%%%%%%%%%%%%%%%%%%%%%%%%%%%%%%%%%%%%%%%%%%%%

\section{Introduction}	\label{sec:introduction}

Let $\Gamma$ be a non-elementary discrete subgroup of $\operatorname{SO}(d, 1)$. Then it acts properly on $\mathbb{H}^{d}$ with basepoint $x_0$. The limit set $\Lambda \Gamma$ of $\Gamma$ is the set of accumulation points of the orbit $\Gamma \cdot x_0$ in $\mathbb{H}^{d} \cup \partial \mathbb{H}^{d}$. It is the smallest nonempty $\Gamma$-invariant closed subset in $\mathbb{H}^{d}$, and captures the dynamics of $\Gamma$ such as growth, equidistribution, etc.

There is a distinguished subset of $\Lambda \Gamma$ that corresponds to geodesic rays that return infinitely often to a compact part of the quotient space $\mathbb{H}^{d} /\Gamma$. Among such rays, some rays forever stay in a compact part of $\mathbb{H}^{d}/\Gamma$. To formulate this, given a boundary point $\xi \in \partial \mathbb{H}^{d}$, let $\gamma : [0, +\infty) \rightarrow \mathbb{H}^{d}$ be the length parametrization of the geodesic ray $[x_{0}, \xi)$ and let \[
f_{\xi}(t) := d_{\mathbb{H}^{d}} \big(\gamma(t), \,\Gamma \cdot x_{0} \big).
\]
We define the conical limit set and the uniformly conical limit set by\[\begin{aligned}
\Lambda_{c} \Gamma &:= \big\{ \xi : \liminf_{t\rightarrow +\infty} f_{\xi}(t) < +\infty\big\},\\
\Lambda_{uc} \Gamma &:= \big\{ \xi : \sup_{t\rightarrow +\infty} f_{\xi}(t) < +\infty\big\}.
\end{aligned}
\]

When $\Gamma$ is cocompact, all of $\Lambda \Gamma$, $\Lambda_{c} \Gamma$ and $\Lambda_{uc}\Gamma$ are equal to the entire boundary $\partial \mathbb{H}^{d}$ and every geodesic ray on $\mathbb{H}^{d}$ is uniformly wrapped in $\mathbb{H}^{d}/\Gamma$, i.e., $f_{\xi}(t)$ is uniformly bounded for $t \in [0, +\infty)$. More generally, we have $\Lambda \Gamma = \Lambda_{c}\Gamma = \Lambda_{uc} \Gamma$ when $\Gamma$ is \emph{convex-cocompact} (see Definition \ref{dfn:convexccpt}). In this case, we have a dichotomy for boundary points: for each $\xi \in \partial \mathbb{H}^{d}$, there exists $C>0$ such that either $f_{\xi}(t)\le C$ for all $t>0$ or $f_{\xi}(t) \ge t-C$ for all $t > 0$. In other words, each ray $\gamma = [x_{0}, \xi)$ on $\mathbb{H}^{d}$ either stays in a bounded neighborhood of the $\Gamma$-orbit (when $\xi \in \Lambda \Gamma$) or moves away from the $\Gamma$-orbit in the fastest way (when $\xi \notin \Lambda \Gamma$). 

When $\Gamma$ is non-elementary but not convex-cocompact, some points of $\Lambda _{c}\Gamma$ correspond to geodesic rays that do not stay in a compact set of $\mathbb{H}^{d}/\Gamma$. In particular, the nonuniformly conical limit set $\Lambda_{nuc} \Gamma := \Lambda_{c}\Gamma \setminus \Lambda_{uc}\Gamma$ is nonempty. One may then ask how large $\Lambda_{uc}\Gamma$ and $\Lambda_{nuc}\Gamma$ are compared to $\Lambda_{c} \Gamma$. 

One way to measure the largeness of a subset of $\partial \mathbb{H}^{d}$ is to compute its \emph{Hausdorff dimension}. We endow $\partial \mathbb{H}^{d}$ with the visual metric \[
d_{vis}(\xi, \eta) := \measuredangle \xi x_0 \eta \quad (\forall \xi, \eta \in \partial \mathbb{H}^{d}).
\]
Now let $A \subseteq \partial \mathbb{H}^{d}$. We can define \[
H_{\delta }^{s}(A)=\inf \left\{\sum _{i=1}^{\infty }(\operatorname {diam} U_{i})^{s}: U_{i} \,\,\textrm{open},\,\,\bigcup _{i=1}^{\infty }U_{i}\supseteq A,\operatorname {diam} U_{i}<\delta \right\}.
\]
Then the Hausdorff dimension of $A$ is defined as\[
\Hdim(A) := \inf \left\{ s : \lim_{\delta \rightarrow 0} H_{\delta}^{s} (A)=0 \right\} = \sup \left\{ s : \lim_{\delta \rightarrow 0} H_{\delta}^{s} (A)=+\infty\right\}.
\]Our goal is to relate the Hausdorff dimensions of $\Lambda_{uc} \Gamma$ and $\Lambda_{nuc} \Gamma$ with another quantity.

A key numerical invariant of $\Gamma$ is its critical exponent, which describes the exponential growth rate of the $\Gamma$-orbit of $x_0$. This is defined via the Poincar{\'e} series \begin{equation}\label{eqn:Poincare}
\mathcal{P}_{\Gamma, x_{0}}(s) := \sum_{g \in \Gamma} e^{-s d(x_{0}, gx_{0})}.
\end{equation}
The critical exponent of $\Gamma$, denoted $\delta_{\Gamma}$, is the unique number such that $\mathcal{P}_{\Gamma, x_{0}} (s)$ diverges for $s < \delta_{\Gamma}$ and converges for $s > \delta_{\Gamma}$. When $\Gamma$ is non-elementary, it is always positive.

S. J. Patterson \cite{patterson1976the-limit} and D. Sullivan \cite{sullivan1979density} developed a powerful theory that relates the critical exponent, conformal densities and the Hausdorff dimension of the conical limit set of Fuchsian groups. This was generalized to non-Fuchsian Kleinian groups by C. J. Bishop and P. W. Jones as follows:

\begin{thm}[{\cite[Theorem 1.1]{bishop1997hausdorff}}]
Let $\Gamma$ be a non-elementary discrete group acting properly on a hyperbolic space $\mathbb{H}^{d}$. Then we have \[
\Hd(\Lambda_{c} \Gamma) = \Hd(\Lambda_{uc} \Gamma) = \delta_{\Gamma}.
\]
\end{thm}

Meanwhile, there are Kleinian groups $\Gamma$ for which nonuniformly conical limit points are generic  in $\Lambda \Gamma$ with respect to the Patterson--Sullivan measure. This motivates the study of how large the nonuniformly conical limit set is.

In \cite{yang2025hausdorff}, M. Mj and W. Yang proved that the Hausdorff dimensions of the  nonuniformly conical limit set and the conical limit set are the same:

\begin{thm}[{\cite[Theorem 1.10]{yang2025hausdorff}}]\label{thm:mainYang}
Let $\Gamma$ be a non-elementary discrete group acting properly on a hyperbolic space $\mathbb{H}^{d}$. Then we have \[
\Hd(\Lambda_{c} \Gamma) = \Hd(\Lambda_{uc} \Gamma) = \Hd(\Lambda_{nuc}\Gamma)= \delta_{\Gamma}.
\]
\end{thm}

In fact, they prove this theorem by focusing on more specific boundary points, namely, the Myrberg limit points. A boundary point $\xi \in \Lambda \Gamma$ is called \emph{Myrberg} if the $\Gamma$-orbit closure of $[x_{0}, \xi)$ contains every geodesic connecting a pair of distinct points $\eta, \zeta \in \Lambda \Gamma$. Equivalently, $\xi \in \Lambda \Gamma$ is Myrberg if there exists a constant $K>0$ such that, for any $g \in \Gamma$ there exists $h \in \Gamma$ such that $g [x_0, hx_0 ]$ is contained in the $K$-neighborhood of $[x_0, \xi)$.

The collection of the Myrberg limit points is denoted by $\Lambda_{Myr}\Gamma$. When $\Gamma$ is not convex-cocompact, every Myrberg limit point is nonuniformly conical, and indeed M. Mj and W. Yang showed that $\Hd(\Lambda_{Myr}\Gamma) = \delta_{\Gamma}$.

Let us define the sublinear growth limit set by \[
\Lambda_{sublinear}\Gamma := \big\{ \xi \in \Lambda_{c} \Gamma: \lim_{t \rightarrow +\infty} f_{\xi}(t)/t  = 0\big\}.
\]
For many non-convex-cocompact groups $\Gamma$, the generic points of $\Lambda \Gamma$ with respect to the Patterson--Sullivan measure are contained in $\Lambda_{sublinear}\Gamma$. In fact, M. Mj and W. Yang's method provides the same estimate as in Theorem \ref{thm:mainYang} for the Hausdorff dimension of $\Lambda_{sublinear}\Gamma \cap \Lambda_{Myr}$.

We now explain our contribution. Given a non-elementary Kleinian group $\Gamma$, we  construct a  subsemigroup of $\Gamma$ that is suited for the Patterson--Sullivan theory. Our main technical result is as follows. 

\begin{theorem}\label{thm:semiconvexConst}
Let $\Gamma$ be a non-elementary discrete group acting properly on a hyperbolic space $\mathbb{H}^{d} \ni x_0$, and let $\epsilon>0$. Then there exists $K>0$ and a free subsemigroup $F \subseteq \Gamma$ such that the following properties hold.
\begin{enumerate}
\item (Divergence) The Poincar{\'e} series $\mathcal{P}_{F}(s) := \sum_{g \in F} e^{-s d(x_0, gx_0) }$ diverges at the critical exponent $\delta_{F}$ of $F$.
\item (Approximation) The inequality $\delta_{F} \ge \delta_{\Gamma} -\epsilon$ holds.
\item (Myrberg property) For any $g \in \Gamma$, there exists $f \in F$ such that the $K$-neighborhood of the geodesic $[x_0, fx_0]$ contains a $\Gamma$-translate of  the geodesic $[x_0, gx_0]$.
\end{enumerate}
\end{theorem}

We then construct the Patterson--Sullivan measure (PS measure) for such subsemigroups $F$. The key point of Theorem \ref{thm:semiconvexConst} is the divergence property (Item 1). The celebrated Hopf--Tsuji--Sullivan dichotomy tells us that the conical limit set of a subgroup of $SO(d, 1)$ is PS-measure-conull if and only if the subgroup is of divergence type. By establishing an analogue  for semigroups, we conclude that the PS measure for $F$ gives full measure to the conical limit set.

Moreover, with respect to this PS measure, generic $F$-limit points will be $\Gamma$-Myrberg and $\Gamma$-sublinearly conical. (This is to be compared with Y. Qing and W. Yang's genericity of sublinearly conical limit points for statistically convex-cocompact group actions \cite{qing2024genericity}.) Based on Nicholls' strategy in \cite{nicholls1989the-ergodic}, we recover M. Mj and W. Yang's estimate of the Hausdorff dimension of $\Lambda_{Myr} \cap \Lambda_{sublinear}$.

\begin{theorem}\label{thm:main}
Let $\Gamma$ be a non-convex-cocompact, non-elementary Kleinian group acting on $\mathbb{H}^{d}$. Then we have \[
\Hdim \big(\Lambda_{Myr} \Gamma \cap \Lambda_{sublinear}\Gamma\big) = \delta_{\Gamma}.
\]
\end{theorem}

Our approach shares some features with that of M. Mj and W. Yang. In particular, we aim to construct \emph{quasi-radial trees} in $\mathbb{H}^{d}$. However, the resulting object is not exactly the quasi-radial tree that Mj and Yang construct. In \cite{yang2025hausdorff}, the authors consider infinite words by drawing the first letter, second letter, etc. from distinct collections. For us, the alphabets for the first letter, second letter, etc. are the same. This homogeneity allows us to develop a Patterson--Sullivan theory.

\begin{remark}
Our restriction to Kleinian groups is merely for convenience. The same technique generalize to isometry groups of Gromov hyperbolic spaces, and more generally, isometry groups of metric spaces with contracting isometries.
\end{remark}

Theorem \ref{thm:main} is proved by K. Falk and K. Matsuzaki for groups with finite Bowen-Margulis-Sullivan measure \cite[Proposition 6.1]{falk2020on-horospheric}. Falk and Matsuzaki's approach is based on the Patterson-Sullivan theory pioneered by S. J. Patterson \cite{patterson1976the-limit} and D. Sullivan \cite{sullivan1979density}, and later adapted by P. J. Nicholls \cite{nicholls1989the-ergodic}. Their method is suited for groups $\Gamma$ for which the $\delta_{\Gamma}$-dimensional conformal measure on $\Lambda \Gamma$ gives full measure to the sublinear growth limit set. It is not known whether it is true for every $\Gamma$ of divergence type. We note D. Sullivan's conjecture after \cite[Corollary 19]{sullivan1979density} in this direction.

\subsection*{Acknowledgements} 
The author is grateful to Dongryul M. Kim for relevant discussions. The author also thanks Or Landesberg for pointing out a mistake in the earlier version of this paper.

The author was supported by the Mid-Career Researcher Program (RS-2023-00278510) through the National Research Foundation funded by the government of Korea, and by the KIAS individual grant (MG091901) at KIAS.

\section{Divergence type}

Throughout the paper, $\Gamma$ denotes a Kleinian group and $x_0$ denotes a basepoint in $\mathbb{H}^{d}$. That means, we always assume that: \[
\textrm{$\Gamma$ is a discrete group properly acting on $\mathbb{H}^{d} \ni x_{0}$}.
\]
Given a pair of points $x, y \in \mathbb{H}^{d}$, we denote by $[x, y]$ the geodesic connecting $x$ to $y$.

Given a set $A$ of isometries of $\mathbb{H}^{d}$, we define the Poincar{\'e} series for $A$ by \[
\mathcal{P}_{A}(s) = \mathcal{P}_{A, x_0}(s) := \sum_{g \in A} e^{-s d(x_0, gx_0)} \quad( 0 \le s < +\infty).
\]
Its abscissa of convergence is called the \emph{critical exponent} of $A$ and is denoted by $\delta_{A}$. Equivalently, we have \[
\delta_{A} = \limsup_{R \rightarrow +\infty} \frac{\log \# \{ g \in A : d(x_0, gx_0 ) < R\}}{R}.
\]
We say that $A$ is of \emph{divergence type} (\emph{convergence type}, resp.) if $\mathcal{P}_{A} (\delta_{A}) = +\infty$ ($\mathcal{P}_{A} (\delta_{A}) < +\infty$, resp.).

\begin{dfn}\label{dfn:convexccpt}
We say that $\Gamma$ is \emph{convex-cocompact} if there exists $K>0$ such that the convex hull of $\Gamma \cdot x_{0}$ is contained in the $K$-neighborhood of $\Gamma \cdot x_{0}$, i.e., for each $g, h \in \Gamma$, the geodesic $[gx_{0}, hx_{0}]$ is contained in $\mathcal{N}_{K}(\Gamma x_{0} )$.
\end{dfn}

In \cite{yang2019statistically}, W. Yang studied the following subset that witnesses the failure of  convex-cocompactness: \[
\mathcal{O}_{M_1, M_2} := \left\{ g \in \Gamma : \begin{array}{c} \exists x, y \in \mathbb{H}^{d}\,\,\textrm{such that}\\  d(x, x_0), d(y, gx_0) \le M_1\,\,\textrm{and}\,\, d([x, y], \Gamma x_0) \ge M_2\end{array}\right\}.
\]
The following lemma justifies the definition. This lemma is implicit in \cite{yang2019statistically}; we record its proof  for completeness.

\begin{lem}\label{lem:nonConvex}
Let $\Gamma$ be a non-convex-cocompact Kleinian group. Then for any $M>0$, $\mathcal{O}_{M, M}$ is infinite. Equivalently, for any $M, R>0$, there exists $g \in \mathcal{O}_{M, M}$ such that $d(x_0, gx_0) \ge R$.\end{lem}

\begin{proof}
Since $\Gamma$ is not convex-cocompact, $\bigcup_{g, h \in \Gamma} [gx_0, hx_0]$ is not contained in the $(2M+R)$-neighborhood of $\Gamma x_0$. Let $g \in \Gamma$ be such that $[x_0, gx_0]$ contains a point $p$ outside $\mathcal{N}_{(2M+R)} (\Gamma x_0)$.

Now, let $P$ be the earliest point on $[x_0, p]$ such that $[P, p]$ is outside the open $M$-neighborhood of $\Gamma x_0$. Let $Q$ be the latest point on $[p, gx_0]$ such that $[p, Q]$ is outside the open $M$-neighborhood of $\Gamma x_0$. It is clear that $d(P, p) > R+M$, as $d(p, \Gamma x_0) > M+R$. Likewise, $d(p, Q) > R+M$.

Furthermore, the distance of $P$ from $\Gamma \cdot x_0$ is exactly $M$; otherwise it would contradict the optimality of $P$. Hence, $d(P, ax_0) = M$ for some $a \in \Gamma$. Similarly, $d(Q, bx_0) = M$ for some $b \in \Gamma$. Note also that $[P, Q]$ does not enter the open $M$-neighborhood of $\Gamma x_0$. Lastly, note that $d(ax_0, bx_0 ) > 2(R+M) - 2M \ge R$. In summary, the desired conclusion is satisfied by \[
x := a^{-1} P, \,\, y := a^{-1} Q, \,\, g := a^{-1} b \qedhere
\]

\end{proof}

Recall that we have fixed $x_0 \in \mathbb{H}^{d}$. We will use the notation \[
\|g \| := d(x_0, gx_0 )
\]
for $g \in \Gamma$. Note that this is a subadditive norm.

\section{Hyperbolic geometry}
Recall that given $x, y, z \in \mathbb{H}^{d}$, we define the Gromov product of $x$ and $y$ with respect to $z$ by \[
\big(x \big| y \big)_{z} := \frac{1}{2} \big[ d(x, z) + d(z, y) - d(x, y) \big].
\]
It is known that $\mathbb{H}^{d}$ is $(\ln 2)$-hyperbolic: 

\begin{lem}[{\cite[Theorem 4.2, 5.1]{nica2016strong}}]\label{lem:Gromov}
Let $x, y, z \in \mathbb{H}^{d}$. Then we have \[
(x | z)_{x_{0}} \ge \min \big( (x|y)_{x_0}, (y|z)_{x_{0}} \big) - \ln 2.
\]
\end{lem}

We use the following notion introduced in \cite{gouezel2022exponential}.

\begin{dfn}[{\cite[Definition 3.6]{gouezel2022exponential}}]
Let $C, D \ge 0$. A sequence of points $(z_0, \ldots, z_n) \in \mathbb{H}^{d}$ is a \emph{$(C, D)$-chain} if \[
\begin{aligned}
(z_{i-1} | z_{i+1})_{z_{i}} &\le C &\quad (0 < i < n),\\
d(z_i, z_{i+1}) &\ge D &\quad (0 \le i < n).
\end{aligned}
\]
\end{dfn}

\begin{lem}\label{lem:chain}
Let $(z_0, \ldots, z_n)$ be a $(C, D)$-chain with $D \ge 2C + 15$. Then $(z_{0} | z_{N})_{z_{i}} < C + 1.5$ for each $i=1, \ldots, N-1$. Moreover, there exist $y_{1}, y_{2}, \ldots, y_{N-1}$ on $[z_{0}, z_N ]$ such that \[\begin{aligned}
d(z_{0}, y_1) \le d(z_{0}, y_2) \le \ldots \le d(z_{0}, y_{N-1}), \\
d(z_{i}, y_{i}) \le C + 6 \quad (i=1, \ldots, N-1).
\end{aligned}
\]
\end{lem}
\begin{proof}
The conclusion that $(z_{0} | z_{N})_{z_{i}} < C+2 \ln 2$ for each $i$ follows from \cite[Lemma 3.8]{gouezel2022exponential}. Since $\triangle z_{0} z_i z_N$ is $(\ln 2)$-slim, its insize is at most $6\ln 2$ (cf. \cite[Proposition III.H.1.22]{bridson1999metric}). Hence, there exists a point $y_{i} \in [z_{0}, z_{N}]$ that is $(C + 8\ln2)$-close to $z_{i}$. The rest follows.
\end{proof}

We say that $\Gamma$ is \emph{non-elementary} if $\#\Lambda \Gamma > 2$. Equivalently, $\Gamma$ is non-elementary if it contains two elements $a$ and $b$ such that \[\begin{aligned}
\sup_{n, m \in \Z} (a^{n} x_0 | b^{m} x_0 )_{x_{0}} &<+\infty,\\
 \lim_{n \rightarrow +\infty} \| a^{n} \|&= \lim_{n \rightarrow +\infty} \|b^{n}\| = +\infty.
 \end{aligned}
\]
By replacing $a$ and $b$ with their suitable powers, we may assume that: \[
|a\|\ge 10^{3}\|b\| > 10^{6}\left(3+ \sup_{n, m \in \Z} (a^{n} x_0 | b^{m} x_0 )_{x_{0}}\right).
\] We now study a property about an element $h$ of $\Gamma$: \begin{equation}\label{eqn:propA}
\begin{aligned}
\textrm{there exist $x_0 = z_0, z_1, \ldots, z_n = hx_0$ such that} \\
\textrm{$(a^{-1} x_0, z_0, \ldots, z_n, hax_0)$ is a $\big(\|a\|, 10^{-6} \|a\|\big)$-chain.}
\end{aligned}
\end{equation}

We claim that:
\begin{fact}\label{fact:reduction}
Let $g \in \Gamma$ be such that $\|g\| \ge \|a\|$. Then at least one of the following 4 elements satisfy Property \ref{eqn:propA}: \[
g, \,\, bg, \,\, gb,\,\, bgb.
\]
\end{fact}
\begin{proof}
Let $C = \sup_{n, m \in \Z}(a^{n} x_0 | b^{m} x_0)_{x_0}$. Gromov's inequality (Lemma \ref{lem:Gromov}) asserts that either $(a^{\mp 1} x_0, x_0, g^{\pm 1}x_0)$ is a $(\|a\|, C)$-chain or $(b^{\mp 1} x_0, x_0, g^{\pm 1} x_0)$ is a $(\|b\|, C+\ln 2)$-chain. Hence, at least one of the following 4 sequences is a $(\|b\|, C + \ln 2)$-chain: \[\begin{aligned}
&(a^{-1} x_0, x_0, gx_0, gax_0); &\quad (a^{-1} x_0, x_0, bx_0, bgx_0, bgax_0); & \\
&(a^{-1} x_0, x_0, gx_0, gbx_0, gbax_0); &\quad (a^{-1} x_0, x_0, bx_0, bgx_0, bgb x_0, bgbax_0).& 
\end{aligned}
\]
When $(a^{-1} x_0, x_0, bx_0, bgx_0, bgax_0)$ is a $(\|b\|, C+ \ln 2)$-chain, Lemma \ref{lem:chain} asserts that $\|bg\| \ge \|b\| + \|g\| - 2(C + 6) \ge \|g\| \ge \|a\|$. It also tells us that  $(a^{-1} x_0, x_0, bgx_0, bgax_0)$ is an $(\|a\|, C + 1.5+ \ln2)$-chain. We can similarly handle the remaining cases.
\end{proof}

We now define a map $\Phi : \Gamma \rightarrow \{ g \in \Gamma : \textrm{Property}\,\, \ref{eqn:propA}\}$ using Fact \ref{fact:reduction}. Here are two remarks.\begin{enumerate}
\item We can and will force that $\Phi(g)=g$ for those $g$ with Property \ref{eqn:propA}.
\item We will largely ignore inputs $\{ g \in \Gamma : \|g\| <\|a\| \}$; for those inputs, we plainly define $\Phi(g) = bab$. 
\end{enumerate}
The map $\Phi$ is finite-to-one, and $\|\Phi(g)\|$ and $\|g\|$ differs by at most $2.5\|a\|$. It can hence be checked that: 

\begin{fact}\label{fact:PhiCritExp}
The critical exponent of the subset $\Phi(\Gamma) \subseteq \Gamma$ is equal to $\delta_{\Gamma}$. 
\end{fact}

\begin{quote}\textbf{
From now on, we fix the aforementioned $a \in G$, the map $\Phi$, and use the notation \[
C := 10^{-3}\|a\| = \frac{1}{1000}d(x_0, ax_0).
\]
}
\end{quote}
Recall that $C > 1000$. Note that: \begin{fact}\label{fact:PropAConcat}
If $g, h \in \Gamma$ satisfy Property \ref{eqn:propA}, then there exist an $(1000C, 0.001C)$-chain of the form \[
\big(a^{-1} x_0,\, \,x_0, \,\,\ldots,\, \,gx_0,\, \,gax_0,\, \,\ldots, \,\,gahx_0, \,\,gahax_0\big).
\]
In particular, $gah$ also satisfies Property \ref{eqn:propA}.
\end{fact}

Lemma \ref{lem:chain} now tells us the following.

\begin{lem}\label{lem:extension}
Let $g_1, \ldots, g_N \in \Gamma$ be elements with Property \ref{eqn:propA}. Then the geodesic $[x_{0}, g_1 a g_2 a \cdots g_N x_{0}]$ has points $p_{1}, q_{1},\ldots, p_{N-1}, q_{N-1}$, in order from closest to farthest from $x_{0}$, such that \[
d(g_1 a\cdots g_i x_{0}, p_{i}) < 0.01C, \,\, d(g_1 a\cdots g_i a x_{0}, q_{i}) < 0.01C. \quad (i=1, \ldots, N-1)
\]
In particular, we have $\|g_1 a g_2 a \cdots g_{N} \| \ge \sum_{i=1}^{N} \|g_{i}\|$.
\end{lem}

\begin{dfn}\label{dfn:extensionFtn}
For each $n > 0$, we define the map $\mathcal{F} = \mathcal{F}_{n} : \Gamma^{n} \rightarrow \Gamma$ by \[
(g_1, \ldots, g_n) \mapsto g_1 ag_2 a \cdots g_n.
\]
By abuse of notation, we sometimes suppress $n$ and denote $\mathcal{F}_n$ by $\mathcal{F}$.
\end{dfn}

Lastly, let us record a consequence of the $(\ln2)$-hyperbolicity of $\mathbb{H}^{d}$.

\begin{lem}\label{lem:fellowTravel}
Let $R>0$ and let $x, y , x', y' \in \mathbb{H}^{d}$ be such that $d(x, x'), d(y, y') < R$. Let $p \in [x, y]$ be a point $R$-far from $x$ and $y$. Then $p$ is contained in the $0.1C$-neighborhood of $[x', y']$.

Moreover, $[x, y]$ is contained in the $R$-neighborhood of $[x', y']$.
\end{lem}

For each $r>0$, we let \[
B_{r} := \{ g \in \Gamma : d(x_{0}, gx_{0}) < r\}.
\]
We say that a set $A \subseteq \Gamma$ is \emph{$r$-separated} if $d(ax_{0}, bx_{0}) > r$ for every pair of distinct elements $a$ and $b$ of $A$.

\begin{dfn}\label{dfn:semiconvex}
Let $A \subseteq \Gamma$. We say that $F$ is \emph{$K$-semiconvex} if, for every $g \in A$ and for every $p \in [x_{0}, gx_{0}]$, there exist $h_{1}, h_{2} \in A\cup \{id\}$ and $c \in B_{K}$ such that \[
g = h_{1} c h_{2}\,\,\textrm{and}\,\, d(p, h_{1}x_{0}) < K.
\]
\end{dfn}

Quasiconvex subgroups are examples of semiconvex subsets. Importantly, our construction of the mapping $\mathcal{F}$ guarantees the following.

\begin{lem}\label{lem:semiconvex}
Let $\mathcal{F} : \Gamma^{n} \rightarrow \Gamma$ be the mapping defined in Definition \ref{dfn:extensionFtn}. Let $R>0$, and  Let \[
A \subseteq \big\{ g \in \Gamma : \textrm{$g$ satisfies Property \ref{eqn:propA} and $\|g\| \le R$}\big\}
\] Then $\cup_{n>0} \mathcal{F}(K^{n})$ is a $(600C + R)$-semiconvex subset of $\Gamma$.
\end{lem}

\begin{proof}
Let $g \in \cup_{n>0} \mathcal{F}(A^{n})$ and let $p \in [x_{0}, gx_{0} ]$. We then have \[
g = g_1 ag_2 a \ldots g_n
\]
for some $g_i \in A$. By Lemma \ref{lem:extension}, there exists $p_1, \ldots, p_{n-1}$, in order from closest to farthest from $x_{0}$, such that \[
d(g_1 a \cdots g_{i} x_{0}, p_i ) < 0.01C. \quad (i=1, \ldots, n-1)
\]
Then we have \[
d(p_i, p_{i+1}) < d(x_{0}, ax_{0} ) + d(x_{0}, g_{i+1} x_{0}) + 0.02C < 1000.02C + R.
\]
Similarly, we have $d(x_{0}, p_1) , d(p_{n-1} x_{0}, gx_{0}) < 1000.02C+R$. For convenience, let $p_{0} := x_0$ and $p_{n} := g x_0$. Then $p$ is at least $(500.01C+ 0.5R)$-close to $p_{i}$ for some $i$. Then we have \[
g = (g_1 a \cdots g_{i} ) \cdot a \cdot (g_{i+1} a \cdots g_{n})\,\,\textrm{and}\,\, d(g_1 a \cdots g_{i} x_0, p) < 600C + R.
\]
(Here, $i$ may be 0 or $n$.) The conclusion follows.
\end{proof}

\begin{lem}\label{lem:semiconvexGrowth}
Let $\Gamma$ be a non-elementary Kleinian group acting on $\mathbb{H}^{d}$. Let $K>0$, and  let $F$ be a $K$-semiconvex subset of $\Gamma$ with critical exponent $\delta_{F} >0$. Then $F$ has purely exponential growth, i.e., there exists $M>0$ such that \[
\frac{1}{M} e^{\delta_{F} r} \le \# \big(B_{r} \cap F \big)
\]for each sufficiently large $r$. In particular, the Poincar{\'e} series $\mathcal{P}_{F}(s)$ for $F$ diverges at $s=\delta_{F}$.
\end{lem}

\begin{proof}
Since the action of $\Gamma$ is proper, $B_{100K}$ has finitely many elements; let $N$ be its number. Also, adding $id$ to $F$ does not alter the critical exponent of $F$ nor the growth of $ \# \big(B_{r} \cap F \big)$, so we will assume $id \in F$.

Let us fix $r>0$. Now let $N$ be an arbitrary integer greater than $1$. We will construct a map \[
\mathcal{F} : g \in B_{Nr} \cap F \mapsto \mathcal{F}(g) = (h_1, h_2) \in \big(B_{(N-1) r}\cap F\big) \times \big(B_{r+3K} \cap F\big).
\]
We take the point $p \in [x_{0}, gx_{0}]$ such that $d(x_{0}, p) = \max \big( (N-1) r - K, d(x_{0}, gx_{0}) \big)$. We then take $h_{1}, h_{2} \in F$ and $c \in B_{K}$ for $p$ as in Definition \ref{dfn:semiconvex}, and define $\mathcal{F}(g) := (h_1, h_2)$. In this case, we have \[\begin{aligned}
d(x_{0}, h_1 x_{0}) &< d(x_{0}, p) + K \le (N-1) r, \\
d(h_1 cx_{0}, gx_{0}) &< d(h_1cx_{0}, h_1 x_{0}) + d(h_1 x_{0}, p) + d(p, gx_{0}) \\
&\le 2K + d(p, gx_{0}) \le r + 3K.
\end{aligned}
\]
Hence, $\mathcal{F}(g)$ belongs to the desired codomain. Moreover, note that the map $g \mapsto (h_1, h_2, c)$ is in fact injective, as $g = h_1 c h_2$. Hence, $\mathcal{F}$ is at most $\#B_{K}$-to-1. We conclude that \begin{equation}\label{eqn:repeatedRExp}
\# \big(B_{Nr} \cap F \big) \le \big( \# B_{K} \big) \cdot \big( \# B_{(N-1) r} \cap F \big) \cdot \big( \# B_{r + 3K} \cap F \big).
\end{equation}
Right now, suppose to the contrary that \[
\# \big(B_{r+3K} \cap F \big) = \lambda \cdot \frac{1}{(\#B_{K})\cdot e^{3\delta_{F} K}}\operatorname{exp}\big(\delta_{F} (r+3K) \big) 
\]
for some $\lambda <1$. Then Inequality \ref{eqn:repeatedRExp} for $N=2, 3, \ldots$ imply that \[\#(B_{Nr} \cap F) \le \lambda^{N-1} \cdot (\#B_{r} \cap F) \cdot e^{\delta_{F}(N-1) r}. \quad (N\ge 2)
\]
This implies that $\delta_{F} \le \lambda \delta_{F}$, a contradiction. Hence, we have \[
\# \big( B_{r+3K} \cap F \big) \ge \frac{1}{(\#B_{K})\cdot e^{3\delta_{F} K}}\operatorname{exp}\big(\delta_{F} (r+3K) \big). \quad (r > 0)
\]
This is the desired bound.
\end{proof}

\section{Construction for Theorem \ref{thm:semiconvexConst}}

In this section, we prove Theorem \ref{thm:semiconvexConst} by constructing a set $F = \cup_{i} F_i$. It is not a semigroup but is merely a countable union of semiconvex subsets $F_{i}  \subseteq \Gamma$. Nonetheless, the translate $a\cdot F$ will be a free semigroup with the same desired properties. 

Throughout, $\Gamma$ is a non-elementary, non-convex-cocompact Kleinian group acting on $\mathbb{H}^{d} \ni x_0$. For convenience, let us enumerate $\Gamma$ by \[
\Gamma= \big\{ \mathfrak{g}_{1}, \mathfrak{g}_{2}, \ldots\big\}.
\]

Given $\epsilon>0$, we will construct semiconvex subsets \[
F_{1} \subsetneq F_{2} \subsetneq \ldots \subseteq \Gamma.
\]
and a nested sequence of intervals $\{I_{i} = [\alpha_i, \beta_i] \}_{i > 0}$, i.e., $I_{1} \supsetneq I_{2} \supsetneq \ldots$, such that the following holds for each $i>0$.

1. $\alpha_{i}$ equals $\delta_{F_{i}}$, the critical exponent of $F_i $. Moreover, $\alpha_{1} \ge (1-\epsilon)\delta_{\Gamma}$ holds.

2. $0 < \beta_{i} - \alpha_{i} \le 2^{-i}$.

3. $\mathcal{P}_{F_{i}}(\beta_{i}) > 2^{i}$.

4. For each $j \le i$, $\mathcal{P}_{F_{i}}(\beta_{j}) \le (2 - 2^{j-i})\cdot \mathcal{P}_{F_{j}}(\beta_{j})$.

We will then let $F := \cup_{i > 0} F_{i}$. Note that for each $s>0$, $\mathcal{P}_{F_{i}}(s) \nearrow \mathcal{P}_{F}(s)$ as $i \rightarrow +\infty$. In particular, for each $i>0$ we have\[
\mathcal{P}_{F}(\alpha_{i}) \ge \mathcal{P}_{F_{i}}(\alpha_{i}) = +\infty
\]
and \[
\mathcal{P}_{F}(\beta_{i}) =\lim_{j \rightarrow +\infty} \mathcal{P}_{F_{j}}(\beta_{i}) \le 2 \mathcal{P}_{F_{i}}(\beta_{i}) < +\infty.
\]
Lastly, we have \[
\mathcal{P}_{F} (\beta_{i}) \ge \mathcal{P}_{F_{i}}(\beta_{i}) > 2^{i}
\]
for each $i$, which implies that $\lim_{s \searrow \lim_{i} \beta_{i}} \mathcal{P}_{F}(s) = +\infty$. In summary, $\mathcal{P}_{F}(s)$ diverges at $s = \lim_{i} \alpha_{i} = \lim_{i} \beta_{i}$, the critical exponent of $F$.

In fact, we will choose a finite set $K \subseteq \Phi(\Gamma)$ and elements $\varphi_1, \varphi_2, \ldots \in \Phi(\Gamma)$ and declare \[
F_{i} := \cup_{n>0} \mathcal{F}^{n}\Big( \big(K \cup \{\varphi_1, \ldots, \varphi_{i-1}\}\big)^{n}\Big).
\]
Let us give more description about $\varphi_{i}$'s. We will define numbers $100<R_{0} < R_{1} < R_{2} < \ldots$ that increase exponentially. Then $\phi_{i} \in \Gamma$ will be chosen such that $\|\phi_{i}\| \ge R_{i}$. We  also consider an enumeration $G = \{\mathfrak{g}_{1}, \mathfrak{g}_{2}, \ldots\}$. Then $\varphi_{i}$ is constructed using $\phi_{i}$ and $\mathfrak{g}_{i}$. As a result, we will also have $\|\varphi_{i}\| \ge R_{i}$.

We now begin the construction by fixing $0<\epsilon < 1$.
Recall that \[
\delta_{\Gamma} := \limsup_{r \rightarrow +\infty} \frac{\log \# B_{r}}{r} = \limsup_{r \rightarrow +\infty} \frac{\log \# \big(B_{r} \cap \Phi(\Gamma) \big)}{r}.
\]
We take $r_0$ such that \[
\#B_{r} \le \operatorname{exp} \big((1+0.001\epsilon)\delta_{\Gamma}r\big)\quad \textrm{for each $r \ge r_{0}$}.
\] Next, we take a large radius $R_{0}> \frac{10^{7}(C + r_{0})(1+\delta_{\Gamma})}{\delta_{\Gamma}\epsilon}$ such that \begin{equation}\label{eqn:RSuffLarge}
\begin{aligned}
\#\big(B_{R_{0}}\cap \Phi(\Gamma)\big) &\ge \operatorname{exp}\big((1- 0.001\epsilon)\delta_{\Gamma}R_{0}\big).
\end{aligned}
\end{equation}
Since $(1-0.003\epsilon)R_{0} \ge 0.5R_{0} > r_{0}$, we have\[\begin{aligned}
\#B_{(1-0.003\epsilon)R_{0}} &\le \operatorname{exp}\big((1 + 0.001\epsilon)(1-0.003\epsilon)\delta_{\Gamma}R_{0}) \\
&\le \operatorname{exp}((1- 0.002\epsilon)\delta_{\Gamma} R_{0}\big).
\end{aligned}\]
This implies that \[\begin{aligned}
\# \big( B_{R_{0}} \cap \Phi(\Gamma) \setminus B_{(1-0.003\epsilon)R_{0}}\big) &\ge \big(1 -e^{-0.001\epsilon \delta_{\Gamma} R_{0} }\big) \operatorname{exp}\big((1 - 0.001\epsilon)\delta_{\Gamma}R_{0}\big) \\
&\ge 0.99\operatorname{exp}\big((1 - 0.001\epsilon)\delta_{\Gamma}R_{0}\big)\\
&\ge \operatorname{exp}((1- 0.002\epsilon)\delta_{\Gamma}R_{0}).
\end{aligned}
\]

Take a maximally $0.004\epsilon R_{0}$-separated subset $K$ of $B_{R_{0}} \cap \Phi(\Gamma) \setminus B_{(1-0.003\epsilon)R_{0}}$. Here, $B_{0.004\epsilon R_{0}}$ has at most $e^{0.005\epsilon R_{0}}$ elements, and $K \cdot B_{0.004\epsilon R_{0}}$ covers $B_{R_{0}} \cap \Phi(\Gamma) \setminus B_{(1-0.003\epsilon)R_{0}}$. This implies that $\#K \ge \operatorname{exp}((1- 0.007\epsilon)\delta_{\Gamma}R_{0})$.

We set $F_{1} := \cup_{n>0} \mathcal{F}_{n}(K^{n})$. Since $K$ is a finite subset of $\Phi(\Gamma)$, Lemma \ref{lem:semiconvex} tells us that $F_{1}$ is semiconvex. We now want to understand the critical exponent $\alpha_{1} := \delta_{F_{1}}$ of $F_{1}$.

\begin{claim}\label{claim:FnInj}
For each $n$, $\mathcal{F}_{n} : K^{n} \rightarrow \Gamma$ is injective. 
\end{claim}

\begin{proof}[Proof of Claim \ref{claim:FnInj}]
To see this claim, suppose that $\mathcal{F}_{n}(g_1, \ldots, g_n) = \mathcal{F}_{n}(h_1, \ldots, h_n) = u$ for some $u \in \Gamma$ and $g_1, \ldots, g_n, h_1, \ldots, h_n \in K$. By the construction and Lemma \ref{lem:extension}, there exist $p \in [x_{0} , ux_{0}]$ that is $C$-close to $g_1 x_{0}$ and $q \in [x_{0}, ux_{0} ]$ that is $C$-close to $h_1 x_{0}$. Recall that $d(x_{0}, g_1 x_{0})$ and $d(x_{0}, h_1 x_{0})$ both lie in the interval $((1-0.003\epsilon)R_{0}, R_{0})$. Hence, they differ by at most $0.003\epsilon R_{0}$. This means that $d(x_{0}, p)$ and $d(x_{0}, q)$ differ by at most $0.003\epsilon R_{0} + 2C$. Since $p$ and $q$ lie on the same geodesic, we deduce that $d(p, q) \le 0.003\epsilon R_{0} + 2C$. This in turn implies that $g_1 x_{0}$ and $h_1x_{0}$ are $(0.003\epsilon R_{0} + 4C)$-close, and hence $0.004\epsilon R_{0}$-close.

Recall that $g_1$ and $h_1$ are drawn from $K$, a $0.004\epsilon R_{0}$-separated set. Hence, the above distance inequality implies $g_1 = h_1$ and $\mathcal{F}(g_2, \ldots, g_n) = \mathcal{F}(h_2, \ldots, h_n)$. 
We inductively conclude that $g_i = h_i$ for each $i$.
\end{proof}

Given the claim, there are $(\#K)^{n} \ge  \operatorname{exp}((1- 0.007\epsilon)\delta_{\Gamma}R_{0}n)$ elements in $\mathcal{F}_{n}(K^{n})$, which is contained in $B_{n(R_{0} + 1000C)} \cap F_{1}$. It follows that \[
\delta_{F_{1}} \ge (1- 0.007\epsilon)\delta_{\Gamma}R_{0} \cdot \frac{1}{R_{0}+1000C} \ge (1-0.008\epsilon) \delta_{\Gamma}.
\]
We finally take $\beta_{1} \in (\alpha_{1}, \alpha_{1} + 0.5)$ that is close enough to $\alpha_{1}$ such that $\mathcal{P}_{F_{1}}(\beta_{1}) > 2$. This concludes the construction of $I_{1} = [\alpha_1, \beta_{1}]$ and $F_{1}$.

Now, having constructed $I_{1} \supsetneq \ldots \supsetneq I_{k}$ and $\varphi_{1}, \ldots, \varphi_{k-1} \in \Phi(\Gamma)$ such that \[
F_{k} := \cup_{n} \mathcal{F}\Big(\big( K \cup \{\varphi_1, \ldots, \varphi_{i-1}\} \big)^{n}\Big)
\]
satisfy Condition 1, 2, 3 and 4 for $i=1, \ldots, k$, we will now construct $I_{k+1}\subsetneq I_{k}$ and $\varphi_{k} \in \Phi(\Gamma)$.

For each $j \le k$, let us denote $M_{j} := \mathcal{P}_{F_{k}} (\beta_{j})$. Recall that Condition 4 for $i=k$ tells us that 
\[
M_{j}\le (2-2^{j-k}) \mathcal{P}_{F_{j}}(\beta_{j})<  (2-2^{j-k-1}) \mathcal{P}_{F_{j}}(\beta_{j}).
\]
Hence, there exists a small enough $0<\epsilon_{j}<1/2$ such that \begin{equation}\label{eqn:MEpsilon}
(M_{j} +2\epsilon_{j} + 4M_{j}\epsilon_{j} )\frac{1}{1-2\epsilon_{j}M_{j}} < (2-2^{j-k-1}) \mathcal{P}_{F_{j}}(\beta_{j}).
\end{equation}
We take sufficiently large $R_{k}>100R_{k-1}$ such that \begin{enumerate}
\item $K \cup \{\varphi_1, \ldots, \varphi_{k-1} \} \subseteq B_{R_{k} - 10^{4} C}$, and 
\item $e^{-\beta_{j} R_{k}} \le \epsilon_{j}$ for each $j \le k$. 
\end{enumerate}
We now take $\phi_{k} \in \mathcal{O}_{C, C} \setminus B_{R_{k} + 2500C}$ using Lemma \ref{lem:nonConvex}. Now, let\[
\varphi_{k} := \Phi(\phi_{k}) \cdot a \cdot \Phi(\mathfrak{g}_{k})
\]
By Fact \ref{fact:PropAConcat}, this is an element of $\Phi(\Gamma)$. Furthermore, Lemma \ref{lem:extension} tells us that $\|\varphi_{k}\| \ge \|\Phi(\phi_{k})\| \ge \|\phi_{k}\| - 2500C \ge R_{k}$.
We then define \[
F_{k+1} := \cup_{n} \mathcal{F}\Big(\big( K \cup \{\varphi_1, \ldots, \varphi_{k}\} \big)^{n}\Big).
\]
As $F_{k+1} \supsetneq F_{k}$, we have $\alpha_{k+1} := \delta_{F_{k+1}} \ge \delta_{F_{k}}=\alpha_{k}$. Since $F_{k+1}$ is semiconvex by Lemma \ref{lem:semiconvex} and has purely exponential growth by Lemma \ref{lem:semiconvexGrowth}, the Poincar{\'e} series $\mathcal{P}_{F_{k+1}}(s)$ diverges at $s = \alpha_{k+1}$. We take \[
\beta_{k+1} \in \Big(\alpha_{k+1}, \, \min \big(\alpha_{k+1} + 2^{-k-1}, \beta_{k} \big) \Big)
\] such that $\mathcal{P}_{F_{k+1}}(\beta_{k+1}) > 2^{k+1}$. 

It remains to check Condition 4 with $i=k+1$. For $j = k+1$ it is clear. Now choose $j < k+1$. To ease the notation, let \[\begin{aligned}
K_{-} &:= K \cup\{\varphi_1, \ldots, \varphi_{k-1}\}, \\
K_{+} &:= K \cup \{\varphi_1, \ldots,\varphi_{k-1}, \varphi_{k}\} = K_{-} \cup \{\varphi_{k}\},\\
F_{k   } &= \cup_{n>0} \mathcal{F}(K_{-}^{n}), \\
\mathcal{B} &:= \cup_{n>0} \mathcal{F}(\{\varphi_k \}^{n} ) = \big\{\varphi_{k} ( a\varphi_{k})^{n} : n \ge 0 \big\}. 
\end{aligned}
\]
We need to evaluate \[
\mathcal{P}_{F_{k+1}}(\beta_{j}) = \sum_{g \in \cup_{n} \mathcal{F}^{n} ( K_{+}^{n})} e^{- \beta_{j} d(x_0, g)}.
\]
To do this, let us consider spaces \[\begin{aligned}
\Omega_{n}^{+} &:= \mathcal{B} \times \left(\prod_{i=1}^{n} \big(  F_{k   } \times \mathcal{B} \big) \right)\times \big(F_{k   } \sqcup \{\ast\} \big),\\
\Omega_{n}^{-} &= F_{k   } \times \left(\prod_{i=1}^{n} \big(  \mathcal{B} \times F_{k   } \big) \right)\times \big(\mathcal{B} \sqcup \{\ast\} \big).
\end{aligned}
\]
for $n=0, 1, \ldots$ and let $\Omega^{\pm} := \sqcup_{n=0}^{\infty} \Omega_{n}^{\pm}$. Let $\Omega := \Omega^{+} \sqcup \Omega^{-}$.

We will define a real-valued function $f$ on $\Omega$ and a mapping $\rho : \Omega \rightarrow G$. Our goal is to show that $\rho(\Omega)$ contains entire $\cup_{n} \mathcal{F}(K_{+}^{n})$, and $f(\omega) \ge e^{-\beta_{j} \|\rho(\omega)\|}$ for some $\omega \in \Omega$ such that $\rho(\omega) \in \cup_{n} \mathcal{F}(K_{+}^{n})$. This will then imply that \[
\sum_{g \in \cup_{n} \mathcal{F}(K_{+}^{n})} e^{- \beta_{j} \|g\|} \le \sum_{\omega \in \Omega} f(\omega).
\]

Let $\omega = (u_0, u_{1}, \ldots, u_{2n-2}, u_{2n-1}) \in \Omega_{n}^{+}$. We then define \[\begin{aligned}
\rho(\omega) &:= \left\{ \begin{array}{cc} u_0 a u_1 \cdots a u_{2n-2} & u_{2n-1} = \ast \\ u_0 a u_1 \cdots a u_{2n-1} & \textrm{otherwise}\end{array}\right.,\quad f(\omega)  &= \operatorname{exp}\left( -\beta_{j} \sum_{l : u_{l} \neq \ast}  \|u_{l}\| \right).
\end{aligned}
\]
We define $\rho$ and $f$ the same way for elements of $\Omega_{n}^{-}$.

Let us now show the desired property. For this, let $g  \in \cup_{n }\mathcal{F}(K_{+}^{n})$:\[
g = g_1 a g_2 \cdots a g_{N}
\]
for some $N \ge 1$, where each $g_{i}$ is drawn from $K_{+} = K_{-}\cup \{\varphi_{k}\}$. We will first describe the case where $g_{1} = \varphi_{k}$.

Let us record when $K_{-}$ and $\varphi_{k}$ alternate, i.e., let \[
\{ i(1) < \ldots < i(T) \} := \{ 1 \le i \le \ldots N-1: \textrm{exactly one of $g_{i}, g_{i+1}$ is $\varphi_{k}$}\}.
\]
For convenience, we set $i(0) := 0$ and $i(T+1) := N$. We then have \[
g = u_0 \cdot \prod_{s=1}^{\lfloor T/2 \rfloor} \big(a u_{2s-1} a u_{2s} \big) \cdot v,
\]
where \[\begin{aligned}
u_{s} &:= g_{i(s) + 1} \prod_{l=i(s)+2}^{i(s+1)} ag_{l}\quad (s=0, \ldots, T), \quad v &:= \left\{ \begin{array}{cc} a u_{T} & \textrm{$T$ is odd} \\ id & \textrm{otherwise}. \end{array}\right.
\end{aligned}
\]
Here, it is clear that $u_0, u_2, \ldots \in \mathcal{B}$ and $u_1, u_3, \ldots \in F_{k   }$. In summary, $g$ equals $\rho(\omega)$ for $\omega = (u_0, u_1, \ldots, u_{T})$ if $T$ is odd and $\omega = (u_0, \ldots, u_T, \ast)$ if $T$ is even. In both cases, we have $\omega \in \Omega_{\lfloor T/2 \rfloor}^{+}$.

Furthermore, Lemma \ref{lem:extension} tells us that \[
\|g\| \ge \sum_{s=0}^{T} \|u_{s}\|. 
\]
It is clear that $f(\omega) \ge e^{-\beta_{j} \|\rho(\omega)\|}$.

When $g_{1} \in K_{-}$, we can similarly describe $g$ in terms of elements of $\Omega_{\lfloor T/2 \rfloor}^{-}$.

It remains to estimate the summation of $f$ over $\Omega$. We have \[
\begin{aligned}
\sum_{\w \in \Omega_{n}^{+}} f(\omega) &= \sum_{g_{0}, g_{1}, \ldots, g_{n}\in \mathcal{B}, h_1, \ldots, h_n \in F_{k   }} e^{-\beta_{j} \|g_{0}\|} \cdot \prod_{l=1}^{n} e^{-\beta_{j} \|g_{l}\|}e^{-\beta_{j} \|h_{l}\|} \cdot \left( 1 + \sum_{h \in F_{k   }} e^{-\beta_{j} \|h\|} \right) \\
&= \mathcal{P}_{\mathcal{B}} (\beta_{j}) \Big( \mathcal{P}_{F_{k   }} (\beta_{j}) \mathcal{P}_{\mathcal{B}} (\beta_{j}) \Big)^{n} \big(1 + \mathcal{P}_{F_{k   }} (\beta_{j}) \big).
\end{aligned}
\]
Summing these up for $n=0, 1, \ldots$, we have \[
\sum_{\w \in \Omega^{+}} f(\omega) = \mathcal{P}_{\mathcal{B}} (\beta_{j})\big(1 + \mathcal{P}_{F_{k   }} (\beta_{j}) \big) \frac{1}{1- \mathcal{P}_{F_{k   }} (\beta_{j}) \mathcal{P}_{\mathcal{B}} (\beta_{j})}.
\]
Similarly we have \[
\sum_{\w \in \Omega^{+}} f(\omega) = \mathcal{P}_{F_{k   }} (\beta_{j})\big(1 + \mathcal{P}_{\mathcal{B}} (\beta_{j}) \big) \frac{1}{1- \mathcal{P}_{F_{k   }} (\beta_{j}) \mathcal{P}_{\mathcal{B}} (\beta_{j})}.
\]
At this point, note that elements $\varphi_k (a\varphi_k )^{n}$ of $\mathcal{B}$ satisfy \[
\| \varphi_k (a\varphi_k )^{n}\| \ge (n+1) \|\varphi_{k}\| \ge (n+1) R_{k}.
\]
by Lemma \ref{lem:extension}. Hence, we can estimate  \[
\mathcal{P}_{\mathcal{B}}(\beta_{j}) = \sum_{n \ge 0} e^{-\beta_{j}\|\varphi_{k} (a \varphi_{k})^{n}\|} \le \sum_{n \ge 0} e^{-(n+1) \beta_{j}\|\varphi_{k}\|}  = \frac{e^{-\beta_{j} R_{k}} }{1-e^{-\beta_{j} R_{k}} } \le 2 \epsilon_{j}.
\]
Recall our choice of $\epsilon_{j}$'s in \ref{eqn:MEpsilon}. We deduce that \[
\sum_{\omega \in \Omega} f(\omega) \le \Big( M_{j} (1 + 2\epsilon_{j}) + 2\epsilon_{j} ( 1 + M_{j}) \Big) \frac{1}{1-2\epsilon_{j} M_{j}} < (2-2^{j-k-1}) \mathcal{P}_{F_{j}}(\beta_{j}).
\]
Condition 4 is now established for $i = k+1 > j$. 

\section{Patterson-Sullivan measure}

So far, we have constructed \[
F := \cup_{n>0} \mathcal{F} \Big( \big(K \cup \{\varphi_1, \varphi_2, \ldots \}\big)^{n} \Big)
\]
so that $\mathcal{P}_{F}(s)$ diverges at $s = \delta_{F}$. We now improve Claim \ref{claim:FnInj} and prove that $F$ behaves like an infinite-rank quasi-tree.

\begin{lem}\label{lem:GromNestingPre}
Let $g, h \in K \cup \{\varphi_1, \varphi_2, \ldots \}$ and suppose that there exists $z \in \mathbb{H}^{d}$ such that \[
d\big(gx_0, [x_0, z]\big), d\big(hx_0, [x_0, z]\big) \le 50C.
\]
Then $g = h$.
\end{lem}

\begin{proof}
Let $p, q \in [x_0, z]$ be such that $d(gx_0, p), d(hx_0, q) \le 50C$. Suppose to the contrary that $g \neq h$. 

Consider first the case that $g, h \in K$. In this case, $\|g\|$ and $\|h\|$ lie between $(1-0.003\epsilon)R_0$ and $R_0$, and differ by less than $0.003\epsilon R_{0}$. This implies that $d(x_0, p)$ and $d(x_0, q)$ differ by less than $0.0035\epsilon R_0$. Since $p$ and $q$ are on the same geodesic $[x_0, z]$, this implies that $d(p, q) < 0.0035\epsilon R_0$ and hence $d(gx_0, hx_0 ) < 0.004\epsilon R_0$. This contradicts the requirement that $K$ is $0.004\epsilon R_0$-separated.

Next, suppose that one of $g, h$ are outside of $K$. By swapping $g$ and $h$ if necessary, this means that there exists $k\ge 1$ such that $g = \varphi_k$ and $h \in K \cup \{\varphi_1, \ldots, \varphi_{k-1}\}$.

For concreteness, let us write \[
\Phi(\phi_{k}) = b_1 \phi_{k} b_{2}
\] 
for some $b_{1}, b_{2} \in \{id, b\}$. Recall that $\|b_1\|, \|b_2 \| \le C$. Recall also that $\|\phi_k \| \ge R_{k} \ge \|h\| + 10000C$. Hence, $\|\Phi(\phi_k)\| \ge \|h\| + 9998C$. Now, by applying Lemma \ref{lem:extension} to \[
\varphi_k = (\Phi(\phi_k)) \cdot a \cdot \Phi(\mathfrak{g}_{k}),
\]
there exists a point $P \in [x_0, \varphi_k x_0 ]$ such that $d(\Phi(\phi_k) x_0, P) \le 0.1C$. Now, Lemma \ref{lem:fellowTravel} for $[x_0,  \varphi_k x_0 ]$ and $[x_0, p]$ tells us that there exists $P' \in [x_0, p]$ such that $d(P, P') \le 50.1C$. At this moment, $P'$ is a point on $[x_0, z]$ with $d(x_0, P') \ge \|\Phi(\phi_k) \| - 51C \ge \|h\| + 9940C$. This implies that $q$ is closer to $x_0$ than $P'$ is. We conclude $q \in [x_0, P']$.

Now let us get back to the property of $\phi_k$. Since $\phi_k \in \mathcal{O}_{10C, 10C}$, there exists a point $x, y \in \mathbb{H}^{d}$ with $d(x_0, x), d(\phi_k x_0, y) =10C$ and such that $d([x, y], \Gamma x_0 ) = 10C$.

Then $b_1 [x, y]$ and $b_1 [x_0, \phi_k x_0]$ have pairwise $10C$-close endpoints. Moreover, $b_1 [x_0, \phi_k x_0]$ and $[x_0, P']$ have pairwise $52C$-close endpoints. Lastly, $q$ is $100C$-far from both $x_0$ and $P'$. Hence, Lemma \ref{lem:fellowTravel} tells us that $q$ is $0.1C$-close to $b_1 [x, y ]$. Thus, $hx_0$ is $2C$-close to $b_1 [x, y]$. This contradict the property of $[x, y]$.

From the above contradictions, we conclude $g=h$.
\end{proof}

\begin{lem}\label{lem:GromNesting}
Let $g_1, g_2, \ldots, g_{n}, h_1, \ldots, h_{m} \in  K \cup \{\varphi_1, \varphi_2, \ldots \}$, let $u = \mathcal{F}(g_1, \ldots, g_n)$, $v = \mathcal{F}(h_1, \ldots, h_m )$ and suppose that $(v h_m^{-1} x_0 | u x_0)_{vx_0} < 10C$. Then we have $n \ge m$ and $g_{i} = h_{i}$ for $i=1, \ldots, m$.
\end{lem}

\begin{proof}
First, note that Lemma \ref{lem:extension} asserts that $(x_0 | vh_{m}^{-1} x_0)_{v x_0} \le 0.1C$ and $x_0 | vx_0)_{v h_m^{-1} x_0} = d(x_0, h_m x_0) - (x_0 | vx_0)_{v h_m^{-1} x_0} \ge 900C$. By Gromov's 4-point inquality (Lemma \ref{lem:Gromov}), we deduce that \begin{equation}
\label{eqn:xuvGrom}(x_0 | ux_0 )_{v x_0} < 10.1C.
\end{equation}

By Lemma \ref{lem:extension}, there exists a point $p \in [x_0, ux_0]$ that is $0.1C$-close to $g_1 x_0$, and $q \in [x_0, vx_0]$ that is $0.1C$-close to $h_1 x_0$. Inequality \ref{eqn:xuvGrom} also guarantees a point $Q \in [x_0, ux_0]$  that is $11C$-close to $vx_0$. Note that $[x_0, Q]$ and $[x_0, vx_0]$ are within Hausdorff distance $11C$ by Lemma \ref{lem:fellowTravel}. 

We now divide into two cases. First, if $d(x_0, Q) \ge d(x_0, p)$, then $p$ belongs to $[x_0, Q]$ and we deduce \[d(g_1x_0, [x_0, vx_0]) \le d(gx_0, p) + d_{Haus}([x_0, vx_0], [x_0, Q]) \le 0.1C + 11C \le 12C.
\] Of course, $d(h_1 x_0, [x_0, vx_0]) ]\le 0.1C$. We now apply Lemma \ref{lem:GromNestingPre} to conclude $g_1 = h_1$.

If $d(x_0, Q) \le d(x_0, p)$, $q \in [x_0, vx_0]$ is $11C$-close to a point in $[x_0, Q] \subseteq [x_0, p]$. Hence, $h_1 x_0$ is $12C$-close to $[x_0, p]$. Of course, $g_1 x_0$ is $0.1C$-close to $[x_0, p]$. Again, Lemma \ref{lem:GromNestingPre} implies that $g_1 = h_1$.

We can run this inductively to prove $g_i = h_{i}$ for all $i \le m$. This also concludes $n \ge m$.
\end{proof}

\begin{lem}\label{lem:GromNestDirect}
Let $g_1, g_2, \ldots, g_{n}, h_1, \ldots, h_{m} \in  K \cup \{\varphi_1, \varphi_2, \ldots \}$, let $u = \mathcal{F}(g_1, \ldots, g_n)$, $v = \mathcal{F}(h_1, \ldots, h_m )$ and suppose that $(x_0 | u x_0)_{vx_0} < 9C$. Then we have $n \ge m$ and $g_{i} = h_{i}$ for $i=1, \ldots, m$.
\end{lem}

\begin{proof}
Under the assumption, we have \[
(x_0 | vh_{m}^{-1} x_0)_{vx_0} = d(x_0, h_{m} x_0 ) - (x_0 | vx_0)_{vh_{m}^{-1}x_0} \ge 90C.
\]
Then Lemma \ref{lem:Gromov} implies that $(v h_m^{-1} x_0 | u x_0)_{vx_0} < 10C$. The remaining follows from Lemma \ref{lem:GromNesting}.
\end{proof}

We now construct Patterson--Sullivan measure. For each $s > \delta_{F}$, let\[
\mu_{x_0}^{s} := \frac{1}{\mathcal{P}_{F}(s)} \sum_{g \in F} e^{-s \|g\| } Dir_{gx_0},
\]
where $Dir_{gx_0}$ denotes the Dirac mass at $gx_0$. Note that $\{\mu_{x_0}^{s}\}$ is a probability measure on a compact set $\Gamma x_0 \cup \partial \mathbb{H}^{d}$. We then take a sequence $s_{n} \searrow \delta_{F}$ such that $\{\mu_{x_0}^{s_n}\}$ converges to a limit probability measure, denoted by $\mu$. We call it the Patterson-Sullivan measure for $F$. Recall that $F$ is of divergence type. This implies that $\mu_{x_0}^{s}(gx_0) \rightarrow 0$ as $s \searrow \delta_{F}$ for each $g \in F$. It follows that $\mu$ is supported on $\partial F \subseteq \partial \mathbb{H}^{d}$ only and $\mu(F \cdot x_0) = 0$.

The key property of the Patterson--Sullivan measure is the shadow principle. Given $y \in \mathbb{H}^{d}$ and $r>0$, Let us define \[
\mathcal{S}(y, r) = \{ z \in \mathbb{H}^{d} \cup \partial \mathbb{H}^{d} : (z | x_0)_{y} \le r\}.
\]
Given $\xi \in \partial \mathbb{H}^{d}$ and $r>0$, let us also define \[
B(\xi, r) := \{ \zeta \in \partial \mathbb{H}^{d} : \measuredangle \zeta x_0 \xi \le r\}.
\]

We now formulate the shadow principle.
\begin{prop}[Shadow principle]\label{prop:Shadow}
Let $\mu$ be the Patterson-Sullivan measure for $F$. Then for $K = e^{\delta_{F} \cdot 10^{7} C}$, we have \[
\frac{1}{K} e^{-\delta_{F} \|g\| } \le \mu \big(\mathcal{S}(gx_0, 8C) \big)\le e^{-\delta_{F} \|g\| }.
\]
for every $g \in F$.
\end{prop}

\begin{proof}
Let us first establish the upper bound. Let \[
F' := \{ h \in F : (hx_0 | x_0)_{gx_0} < 8C\} = \{ h \in F : hx_0 \in \mathcal{S}(gx_0, 8C)\}.
\]
Lemma \ref{lem:GromNestDirect} tells us that the map $\Psi : h \mapsto (ga)^{-1} h$ is a one-to-one map from $F'$ into $F$. Furthermore, since $g$ and $(ga)^{-1} h$ both satisfy Property \ref{eqn:propA}, Lemma \ref{lem:extension} tells us that \[
\|h\| \ge \|g\| + \|(ga)^{-1} h\|.
\]

Now, for each $s > \delta_{F}$, note that \[
\begin{aligned}
\mathcal{P}_{F}(s)\cdot \mu_{x_0}^{s} \big(\mathcal{S}(gx_0, 8C) \big) &= \sum_{h \in F'} e^{-s \|h\|} \le \sum_{k \in \Psi (F')} e^{- s \|k\|} e^{-s \|g\|} \\
&\le e^{-s \|g\|} \sum_{k \in F} e^{-s\|k\|} \le \mathcal{P}_{F}(s)\cdot e^{-s \|g\|}.
\end{aligned}
\]
By inserting $s = s_{n}$ as in the construction of the Patterson--Sullivan measure and by taking the limit, we conclude the upper bound.

Next, note that the map $\Psi$ is in fact surjective. Indeed, for any $k \in F$ we have $ga k \in F'$ by Lemma \ref{lem:extension} and $\Psi(gak) = k$. Moreover, note that \[
\|h\| \le \|g\| + \|a\|+  \|(ga)^{-1} h\| \le \|g\| +  \|(ga)^{-1} h\| + 1000C.
\]
Hence, we have \[
\begin{aligned}
\mathcal{P}_{F}(s)\cdot \mu_{x_0}^{s} \big(\mathcal{S}(gx_0, 8C) \big) &= \sum_{h \in F'} e^{-s \|h\|} \ge \sum_{k \in \Psi (F')} e^{- s \|k\|} e^{-s (\|g\|+1000C)} \\
&= e^{-s (\|g\|+1000C)} \sum_{k \in F} e^{-s\|k\|} \le \mathcal{P}_{F}(s)\cdot e^{-s (\|g\|+1000C)}.
\end{aligned}
\]
By inserting $s = s_{n}$ as in the construction of the Patterson--Sullivan measure and by taking the limit, we conclude the lower bound.
\end{proof}

Our next goal is to prove that $\mu$ is fully supported on the conical limit set. When $F$ is replaced with discrete subgroups $G$ of $\Isom(\mathbb{H}^{d})$, this is part of the classical Hopf--Tsuji--Sullivan dichotomy. There, the Patterson--Sullivan measure for $G$ is fully supported on the conical limit set (non-conical limit set, resp.) if and only if $G$ is of divergence type (convergence type, resp.). For us, since $F$ is not a subgroup, we prove it from the scratch.

\begin{lem}\label{lem:conicalFull}
Let $\mu$ be the Patterson-Sullivan measure for $F$. Then we have \[
\mu \left( \left\{ \xi \in \Lambda F: \liminf_{t} d(\gamma(t), F x_0 ) > 3C\,\,\textrm{for}\,\, \gamma = [x_0, \xi) \right\} \right) = 0.
\]
\end{lem}

\begin{proof}
Given $g \in F$, we consider the following property for $\xi \in \partial \mathbb{H}^{d}$: \[
P_{g} := ``\begin{array}{c}\textrm{ there exists $\{z_n \}_{n>0} \subseteq Fx_0 $ such that}\\
               \textrm{$z_{n} \rightarrow \xi$ and $(gx_0 | z_n)_{gax_0} < 2C$.} \end{array}"
\]
Now for $\xi \in \partial \mathbb{H}^{d}$, if it satisfies $P_{g}$ for infinitely many $g \in F $, then clearly $d([x_0, \xi), gx_0) \le 2C$ for infinitely many such $g$'s. Conversely, if $ \liminf_{t} d(\gamma(t), F x_0 ) > 3C$ for $\gamma = [x_0, \xi)$, then $\xi$ satisfies $P_{g}$ for only finitely many $g \in F$. Furthermore, Lemma \ref{lem:GromNestDirect} tells us that, if $\xi$ satisfies $P_{g}$ and $P_{h}$ for two distinct $g, h \in F$, then one is an initial section of the other. In particular, among $g$'s for which $P_{g}(\xi)$ holds, there exists a unique $g$ with maximal $\|g\|$. We call it $g_{\xi}$. Hence, we have a countable measurable partition \[
\left\{ \xi \in \Lambda F : \liminf_{t} d(\gamma(t), F x_0 ) > 2C\,\,\textrm{for}\,\, \gamma = [x_0, \xi) \right\}  \subseteq  \sqcup_{g \in F} \{ \xi : g_{\xi} = g\} \sqcup \{\xi : \not\exists g [P_{g}(\xi)] \}.
\]

Now suppose to the contrary that $\mu$ charges positive weight on the set on the left. Then either $\{\xi : g_{\xi} = g\}$ has positive $\mu$-weight for some $g$, or $\{\xi : \not\exists g [P_{g}(\xi)]\}$ gets positive $\mu$-weight.

Case I) $\{\xi : g_{\xi} = g\}$ has positive $\mu$-weight for some $g \in F$.

Let $h \in F$ be an arbitrary element. Then $\{\xi : g_{\xi} = g\}$ and $\{\xi : g_{\xi} = ha g\}$ are clearly distinct. We now claim: \begin{claim}
\[
\mu (\{\xi : g_{\xi} = ha g\}) \ge e^{-10^{4}\delta_{F} C} \cdot e^{- \delta_{F} \|h\|} \cdot \mu (\{\xi : g_{\xi} = g\}).
\]
\end{claim} 
We use the multiplication map by $ha$: this maps open neighborhoods of $\{\xi : g_{\xi} = g\}$ to open neighborhoods of $\{\xi : g_{\xi} = hag\}$, and vice versa. Hence, if we prove that $\mu(ha O) \ge e^{-10^{4}\delta_{F} C} \cdot e^{- \delta_{F} \|h\|} \mu (O)$ for every open neighborhood $O \subseteq \mathbb{H}^{d}\cup \partial \mathbb{H}^{d}$ of $\{\xi : g_{\xi} = g\}$, then the result follows from outer regularity of $\mu$. This is easily checked by \[
\sum_{k\in F: kx_0 \in O} e^{-s \| ha \cdot k \|} \ge \sum_{k\in F: kx_0 \in O} e^{-s \| h\| + s\|a\| + s\|  k \|} \ge e^{-s \|h\|}  e^{-1000C s}\sum_{k\in F: kx_0 \in O} e^{-s \|k\|}
\]
for each  $s > \delta_{F}$.

Furthermore, clearly $\{\xi : g_{\xi} = ha g\}$ are disjoint for distinct $h$'s. Hence, we have \[
\begin{aligned}
\mu(\Lambda F) &\ge \sum_{h \in F} \mu (\{\xi : g_{\xi} = ha g\})  \ge e^{-10^{4}\delta_{F} C} \sum_{h \in F} e^{- \delta_{F} \|h\|} \cdot \mu (\{\xi : g_{\xi} = g\}) \\
&\ge +\infty \cdot \mu (\{\xi : g_{\xi} = g\}) = +\infty.
\end{aligned}
\]
This is a contradiction.

Case II) $\{\xi \in \Lambda F: \not\exists g [P_{g}(\xi)] \}$ has positive $\mu$-weight.

In this case, fix an arbitrary $h \in F$. Now for $\zeta \in \{\xi \in \Lambda F: \not\exists g [P_{g}(\xi)] \}$, pick $\{z_n\}_{n >0} \subseteq F x_0$ that tends to $\zeta$. Note that $\{ ha z_n \}_{n>0}$ is a sequence tending to $ha \zeta$. Furthermore, $(hx_0 | z_n)_{hax_0} < 2C$ for each $n$ by Lemma \ref{lem:extension}. In summary, $ha\zeta$ satisfies $P_{h}$. Meanwhile, if $ha \zeta$ also satisfies $P_{h'}$ for some $h' \in F$ with $\|h'\| > \|h\|$, then $h$ is an initial section of $h'$ and $\zeta$ satisfies $P_{(ha)^{-1} h'}$. This contradicts the nature of $\zeta$. Hence, $h$ satisfies the maximality condition and $g_{ha\zeta} = h$.

In summary, we have $\{ \xi : g_{\xi} = h\} \supseteq ha \{\xi \in \Lambda F: \not\exists g [P_{g}(\xi)] \}$. As we argued in the previous case, the $\mu$-value of these two sets differ by at most $e^{\delta_{F} (\|h\| + 10^4 C)}$ factor. Hence, $ \{ \xi : g_{\xi} = h\}$ has positive $\mu$-value, and we are reduced to Case I.

Considering these contradictions, we conclude the desired statement.
\end{proof}

Let us now fix an arbitrary $\varphi \in K \cup \{\varphi_1, \varphi_2, \ldots\}$. Our next claim is:

\begin{lem}\label{lem:conicalPhiFull}
Let $\mu$ be the Patterson-Sullivan measure for $F$. Then for $\mu$-a.e. boundary point $\xi \in \Lambda F$, $\xi$ satisfies $P_{g a\varphi}$ for infinitely many $g \in F$. 
\end{lem}

\begin{proof}
The proof is similar to the previous one. If a given $\xi \in \Lambda F$ satisfies $P_{ga \varphi}$ for only finitely many $g \in F$, there exists a unique such $g$ with maximal $\|g\|$. We call it $g_{\xi}$. We have a countable measurable partition \[
\left\{ \xi \in \Lambda F : \textrm{$P_{ga \varphi}(\xi)$ holds for only finitely many $g \in F$}\right\}  \subseteq  \sqcup_{g \in F} \{ \xi : g_{\xi} = g\} \sqcup \{\xi : \not\exists g [P_{g}(\xi)] \}.
\]
As in Case I of the proof of Lemma \ref{lem:conicalFull}, we can observe that $\{\xi : g_{\xi} = g\}$ is $\mu$-null for each $g \in F$. Moreover, as in Case II of the proof of Lemma \ref{lem:conicalFull}, we can observe that $\{\xi : \nexists g [P_{g}(\xi)]\}$ is also $\mu$-null. 
\end{proof}

In particular, some translate of $[x_0, \varphi x_0 ]$ is contained in the $3C$-neighborhood of $[x_0, \xi)$. Moreover, recall that $[x_0, \mathfrak{g}_{k} x_0]$ is contained in the $10C$-neighborhood of $[x_0, \varphi_{k} x_0]$ for each $k$. We conclude that:

\begin{cor}\label{cor:MyrbergGeneric}
Let $\mu$ be the Patterson-Sullivan measure for $F$. Then $\mu$-a.e. limit points are $\Gamma$-Myrberg limit points. 

More specifically, for $\mu$-a.e. $\xi \in \partial F$, for each $g \in \Gamma$, there exists $h \in \Gamma$ such that the $13C$-neighborhood of $[x_0, \xi)$ contains $h [x_0, gx_0]$.
\end{cor}

We lastly investigate the sublinearly conical limit set. Given $\eta>0$, let us define: \[
\Lambda_{\eta} F := \left\{ \xi \in \Lambda_{c} F : \begin{array}{c}\textrm{$\exists T$ such that, for $\gamma = [x_0, \xi)$ and for each $t >T$,}\\
\textrm{ $\gamma([t, t+\eta t])$ is $C$-close to $Fx_0$}\end{array}\right\}.
\]
Every element $\xi$ of $\Lambda_\eta F$ satisfies that $\liminf_{t} d(\gamma(t), Fx_0) / t \le \eta$. 

We also define \[
\mathscr{S}_{\eta, R} := \big\{ \mathcal{S}(ga hx_0, 8C) : g, h \in F, \|h\| > \eta \|g\|, R \le  \|g\| \le R+1\big\} \quad (R=1, 2, \ldots)
\]
and define $\mathscr{S}_{\eta} := \cup_{R>0} \mathscr{S}_{\eta, R}$. Observe that:

\begin{lem}\label{lem:sublinearCon}
Let $\xi \in \Lambda_{c} F \setminus \Lambda_{\eta} F$. Then $\xi$ belongs to infinitely many shadows $\mathcal{S} \in \mathscr{S}_{\eta}$.
\end{lem}

\begin{proof}
Since $\xi \in \Lambda_{c} F$, there exists an infinite sequence $h_1, h_2, \ldots \in K \cup \{\varphi_1, \varphi_2, \ldots\}$ such that $h_1 a \cdots h_{i-1} a h_{i}x_0$ converges to $\xi$. Furthermore, there exist points $p_{i} = \gamma(t_{i})  \in [x_0, \xi)$ that are $C$-close to $h_1 a \cdots h_{i-1} a h_{i} x_0$, respectively. Here, $\{t_{i}\}_{i>0}$ is an increasing sequence that tends to infinity.

Since $\xi \notin \Lambda_{\eta} F$, there must exist infinitely many $i$'s for which \[
\frac{t_{i+1} - t_{i} + 2C}{t_{i} - 2C} > 2\eta.
\]
For such $i$'s, $\xi$ belongs to $\mathcal{S}( h_1 a \cdots h_{i} a h_{i+1} x_0, 8C)$, where $\|h_{i+1}\| > (t_{i+1} - t_{i}) - 2C > 4\eta (t_{i} + 2C) > 4\eta \|h_1 a \cdots h_{i} \|$ if $i$ is sufficiently large. Such shadows belong to $\mathscr{S}_{\eta}$.
\end{proof}

Let us now estimate the number and the $\mu$-sizes of shadows $\mathcal{S}(ga hx_0, 8C)$ in $\mathscr{S}_{\eta, R}$, for large enough $R$. There are at most $\operatorname{exp}\big(\delta_{F}(1+ 0.001\eta) R\big)$-many candidates for $g$. For $h$, we can choose $\varphi_{k}, \varphi_{k+1}, \ldots$ where $k$ is the smallest index such that $\|\varphi_{k}\| \ge \eta R$. Recall that $\|\varphi_{k}\|, \|\varphi_{k+1}\|$, $\ldots$ grow exponentially by the factor of $10$, and in particular $\|\varphi_{k+l}\| \ge 10^{l} \eta R$. Now, given such $g$ and $h=R_{k+l}$ $(l\ge 0)$, Proposition \ref{prop:Shadow} tells us that \[
\mu\big(\mathcal{S}(gah x_0, 8C) \big) \le e^{-\delta_{F} \|g a h\|} \le e^{-\delta{F} (R + \eta R_{k+l})} \le e^{-\delta_{F}R(1+ \eta \cdot  10^{l})}. 
\]
It follows that \[\begin{aligned}
\sum_{\mathcal{S} \in \mathscr{S}_{\eta, R}} \mu(\mathcal{S}) &\le \sum_{l = 0}^{\infty} \operatorname{exp}\big(\delta_{F}(1+ 0.001\eta) R\big) \cdot \operatorname{exp}\big(-\delta_{F}R (1 + \eta \cdot 10^{l})) \\
&\le \sum_{l=0}^{\infty} \operatorname{exp}\big(-0.5 \delta_{F} \eta R \cdot 10^{l}) \le 1.1 \operatorname{exp}\big(-0.5 \delta_{F} \eta R).
\end{aligned}
\]
This is summable over $R$, and $\sum_{\mathcal{S} \in \mathscr{S}_{\eta}} \mu(\mathcal{S})$ is finite. By the Borel-Cantelli Lemma, $\mu$-a.e. limit point $\xi$ is not contained in $\mathcal{S} \in \mathscr{S}_{\eta}$ infinitely often. In other words, $\liminf_{t} d(\gamma(t), F x_0) / t > \eta$ for $\gamma = [x_0, \xi)$ for such $\xi$. Since $\eta$ is arbitrary, we conclude that $\mu$-a.e. limit points are sublinearly conical.

Given the shadow lemma and the $\mu$-genericity of sublinearly conical limit set, we can now adopt Nicholls' argument in \cite[Lemma 9.3.4]{nicholls1989the-ergodic}. For completeness, we include the proof below.

\begin{lem}
Let $A \subseteq \Lambda F$ be a subset with positive $\mu$-value. Then we have $\Hdim(A) \ge \delta_{F}$.
\end{lem}

\begin{proof}
Let $\eta>0$ be a small positive number. Let us stratify $\Lambda_{\eta} F$ into \[
\Lambda_{\eta, T}(F) := \left\{ \xi \in \Lambda F : \begin{array}{c}\textrm{for $\gamma = [x_0, \xi)$ and for each $t >T$,}\\
\textrm{ $\gamma([t, t+\eta t])$ is $C$-close to $Fx_0$}\end{array}\right\}.
\]
Then $\{ \Lambda_{\eta, T}F \cap A\}_{T> 0}$ is an increasing family of compact subsets of $\Lambda F$ whose union contains $A \cap \Lambda_{\eta}F$. This set has positive $\mu$-value, as $A$ has positive $\mu$-value and $\Lambda_{\eta}F$ is $\mu$-conull. Hence, there exists $T_{0}$ such that $\mu(\Lambda_{\eta, T}F \cap A) > 0$. 

Now, for each $\xi \in \Lambda_{\eta, T}F$ and for each $t > 2T$, we claim that $B(\xi, C e^{-t})$ is contained in $\mathcal{S}(gx_0, 8C)$ for some $g \in F$ with $\|g\| \ge (1-\eta)t$. Namely, let us pick $g \in F$ and $\tau \in \big[ (1-\eta) t, t \big]$ such that $d(gx_0, \gamma(\tau)) < C$; such $\tau$ and $g$ exist as $t > 2T$ and $\xi \in \Lambda_{\eta, T} F$. Now pick an arbitrary $\zeta \in B(\xi, Ce^{-t})$. Let $p \in [x_0, \zeta)$ be such that $d(x_0, p) = \tau$. Then $\triangle p x_0 \gamma(\tau)$ is an isosceles triangle with $\measuredangle p x_0 \gamma(\tau) \le C e^{-t}$. By hyperbolic geometry of $\mathbb{H}^{2}$, $p$ and $\gamma(\tau)$ are $C$-close. This implies that $\zeta$ is contained in $\mathcal{S}(gx_0, 2C)$ as desired.

By the Shadow principle (Proposition \ref{prop:Shadow}), we conclude that: \begin{obs}
For each $\xi \in \Lambda_{\eta, T}$ and $0< r < \frac{1}{C} e^{-2T}$, we have  \[
\mu\big( B(\xi, r) \cap \Lambda_{\eta, T} F \big) \le (r/C)^{\delta_{F} (1-\eta)}.
\]
\end{obs}
By the proof of \cite[Theorem 9.3.5]{nicholls1989the-ergodic}, we conclude that $\Lambda_{\eta, T} F \cap A$ has Hausdorff dimension at least $\delta_{F}(1-\eta)$. Hence, the Hausdorff dimension of $A$ is at least $\delta_F (1-\eta)$. Since this is the case for all small enough $\eta$, we conclude that $\Hdim(A) \ge \delta_{F}$. 
\end{proof}

In particular, the set of all Myrberg, sublinearly conical limit points of $\Gamma$ has Hausdorff dimension at least $\delta_{F}$. Finally recall that $\delta_{F} \ge \delta_{\Gamma} - \epsilon$. Since $\epsilon$ is arbitrary, we get $\Hdim(\Lambda_{Myr} \Gamma \cap \Lambda_{sublinear}\Gamma) \ge \delta_{\Gamma}$. Meanwhile, by \cite[Corollary 8.3.2]{nicholls1989the-ergodic}, the conical limit set of $\Gamma$ has Hausdorff dimension $\le \delta_{\Gamma}$. This ends the proof of Theorem \ref{thm:main}.

%%%%%%%%%%%%%%%%%%%%%%%%%%%%%%%%%%%%%%%
%
%							References
%
%%%%%%%%%%%%%%%%%%%%%%%%%%%%%%%%%%%%%%%

\medskip
\bibliographystyle{alpha}
\bibliography{myrb}

\end{document}